\documentclass[12pt,reqno]{article}

\usepackage[usenames]{color}
\usepackage{amssymb}
\usepackage{graphicx}
\usepackage{amscd}

\usepackage{amsthm}
\newtheorem{theorem}{Theorem}

\newtheorem{proposition}[theorem]{Proposition}
\newtheorem{corollary}[theorem]{Corollary}

\theoremstyle{definition}

\newtheorem{example}[theorem]{Example}

\usepackage[colorlinks=true,
linkcolor=webgreen, filecolor=webbrown,
citecolor=webgreen]{hyperref}

\definecolor{webgreen}{rgb}{0,.5,0}
\definecolor{webbrown}{rgb}{.6,0,0}

\usepackage{color}

\usepackage{float}

\usepackage{graphics,amsmath,amssymb}
\usepackage{amsfonts}
\usepackage{latexsym}
\usepackage{epsf}

\newcommand{\seqnum}[1]{\href{http://www.research.att.com/cgi-bin/access.cgi/as/~njas/sequences/eisA.cgi?Anum=#1}{\underline{#1}}}

\begin{document}

\begin{center}
\vskip 1cm{\LARGE\bf Eulerian polynomials as moments, via exponential Riordan arrays} \vskip 1cm \large
Paul Barry\\
School of Science\\
Waterford Institute of Technology\\
Ireland\\
\href{mailto:pbarry@wit.ie}{\tt pbarry@wit.ie} \\

\end{center}
\vskip .2 in

\begin{abstract} Using the theory of exponential Riordan arrays and orthogonal polynomials, we demonstrate that the ``descending power'' Eulerian polynomials, and their once shifted sequence, are moment sequences for simple families of orthogonal polynomials, which we characterize in terms of their three-term recurrence. We obtain the generating functions of the polynomial sequences in terms of continued fractions, and we also calculate their Hankel transforms.
\end{abstract}

\section{Introduction}
The Eulerian polynomials \cite{Comtet, Foata, Hirzebruch, Luschny}
$$P_n(x)=\sum_{k=0}^n W_{n,k} x^k$$ form the sequence ${P_n(x)}$ which begins
$$P_0(x)=1, P_1(x)=1, P_2(x)=1+x,P_3(x)=1+4x+x^2,\ldots,$$ with the well-known triangle of Eulerian numbers \cite{Concrete}
\begin{displaymath}\left(\begin{array}{ccccccc} 1 & 0 & 0
& 0 & 0 & 0 & \ldots \\1 & 0 & 0 & 0 & 0 & 0 & \ldots \\
1 & 1 & 0 & 0 & 0 & 0 & \ldots \\ 1 & 4 & 1 & 0
& 0 & 0 & \ldots \\ 1 & 11 & 11 & 1 & 0 & 0 & \ldots
\\1 & 26 & 66 & 26 & 1 & 0 &\ldots\\ \vdots & \vdots & \vdots & \vdots & \vdots
& \vdots & \ddots\end{array}\right) \end{displaymath}
as coefficient array. These coefficients $W_{n,k}$ obey the recurrence \cite{Hirzebruch}
$$W_{n,k}=(k+1) W_{n-1,k}+(n-k) W_{n-1, k-1}$$ with appropriate boundary conditions. The closed form expression
$$W_{n,k}=\sum_{i=0}^{n-k} (-1)^i \binom{n+1}{i} (n-k-i)^n$$ holds. The polynomials ${P_n(x)}$ were introduced by Euler \cite{Euler} in the form
$$\sum_{k=0}^{\infty}(k+1)^n t^k = \frac{P_n(t)}{(1-t)^{n+1}}.$$
They have exponential generating function
$$\sum_{n=0}^{\infty} P_n(x) \frac{t^n}{n!} = \frac{(1-x)e^{(1-x)t}}{1-x e^{(1-x)t}}.$$
We have $$P_n(x)=\sum_{k=0}^n A_{n,k}x^{n-k},$$ and hence we can regard them as ``descending power'' Eulerian polynomials.

In this note we show that the sequence of Eulerian polynomials $P_n(x)$ is the moment sequence of a family of orthogonal polynomials. In addition, we show that the sequence of shifted Eulerian polynomials $P_{n+1}(x)$ is similarly the moment sequence of a family of orthogonal polynomials.
For this, we will require three results from the theory of exponential Riordan arrays (see Appendix for an introduction to exponential Riordan arrays). These are \cite{Barry_Meixner, Barry_Moment}
\begin{enumerate}
\item The inverse of an exponential Riordan array $[g,f]$ is the coefficient array of a family of orthogonal polynomials if and only if the production matrix of $[g,f]$ is tri-diagonal;
\item If the production matrix of $[g,f]$ is tri-diagonal, then the elements of the first column of $[g,f]$ are the moments of the corresponding family of orthogonal polynomials;
\item The bivariate generating function of the production matrix of $[g,f]$ is given by
$$e^{xy}(Z(x)+A(x)y)$$ where
$$A(x)=f'(\bar{f}(x)),$$ and
$$Z(x)=\frac{g'(\bar{f}(x))}{g(\bar{f}(x))},$$ where $\bar{f}(x)$ is the compositional inverse (series reversion) of
$f(x)$.
\end{enumerate}
A quick introduction to exponential Riordan arrays can be found in the Appendix to this note. For general information on orthogonal polynomials and moments, see \cite{Chihara, Gautschi, Szego}. Continued fractions will be referred to in the sequel; \cite{Wall} is a general reference, while \cite{Kratt, Kratt1} discuss the connection of continued fractions to orthogonal polynomials, moments and Hankel transforms \cite{Layman, Radoux}. We recall that for a given sequence $a_n$ its Hankel transform is the sequence of determinants $h_n=|a_{i+j}|_{0 \le i,j \le n}$. Many interesting examples of number triangles, including exponential Riordan arrays, can be found in Neil Sloane's On-Line
Encyclopedia of Integer Sequences \cite{SL1, SL2}. Sequences are frequently referred to by their
OEIS number. For instance, the binomial matrix (Pascal's triangle) $\mathbf{B}$ with $(n,k)$-th element $\binom{n}{k}$ is \seqnum{A007318}.
\section{The Eulerian polynomials $P_n(x)$}
We consider the sequence with e.g.f.
$$\frac{(\alpha-\beta)e^{(\alpha-\beta)t}}{\alpha-\beta e^{(\alpha-\beta)t}}.$$ This is the sequence that begins
$$1, \alpha, \alpha(\alpha+\beta), \alpha(\alpha^2+4\alpha \beta+\beta^2), \alpha(\alpha^3+11\alpha^2 \beta+11\alpha \beta^2+\beta^3),\ldots.$$ Setting $\alpha=1$ and $\beta=x$ gives us the Eulerian polynomials $P_n(x)$.
We have the
\begin{proposition} The production matrix of the exponential Riordan array
$$\left[\frac{(\alpha-\beta)e^{(\alpha-\beta)t}}{\alpha-\beta e^{(\alpha-\beta)t}}, \frac{e^{(\alpha-\beta)t}-1}{\alpha-\beta e^{(\alpha-\beta)t}}\right]$$ is tri-diagonal.
\end{proposition}
\begin{proof}
Writing the above exponential Riordan array as $[g,f]$, we have
$$f(t)=\frac{e^{(\alpha-\beta)t}-1}{\alpha-\beta e^{(\alpha-\beta)t}}$$ and hence
$$f'(t)=\frac{e^{(\alpha+\beta)t}(\alpha-\beta)^2}{\beta e^{\alpha t}-\alpha e^{\beta t}},$$ and
$$\bar{f}(t)=\frac{1}{\alpha-\beta}\ln\left(\frac{\alpha t+1}{\beta t+1}\right).$$
Then $$A(t)=f'(\bar{f}(t))=(\alpha t+1)(\beta t+1)=1+(\alpha+\beta)t+ \alpha \beta t^2.$$
We have $$ g(t)=\frac{(\alpha-\beta)e^{(\alpha-\beta)t}}{\alpha-\beta e^{(\alpha-\beta)t}}$$ and
hence $$g'(t)=\frac{\alpha e^{(\alpha+\beta)t}(\alpha-\beta)^2}{(\beta e^{\alpha t}-\alpha e^{\beta t})^2},$$
and so $$Z(t)=\frac{g'(\bar{f}(t))}{g(\bar{f}(t))}=\alpha (\beta t+1)=\alpha+\alpha \beta t.$$
Thus the production matrix, which has bivariate g.f. given by
$$ e^{t y}(\alpha+\alpha \beta t +(1+(\alpha+\beta)t+ \alpha \beta t^2)y),$$ is tri-diagonal.
\end{proof}
We note that the production matrix takes the form
\begin{displaymath}\left(\begin{array}{ccccccc}
\alpha
& 1 & 0 & 0 & 0 & 0 & \ldots \\\alpha \beta  & 2\alpha+\beta & 1 & 0 & 0 & 0 &
\ldots \\ 0 & 4\alpha \beta & 3\alpha+2 \beta & 1 & 0 & 0 & \ldots \\ 0 & 0 & 9 \alpha \beta & 4\alpha+3\beta  & 1
&
0 & \ldots \\ 0 & 0 & 0 & 16\alpha \beta & 5\alpha+4 \beta & 1 & \ldots \\0 & 0  & 0 & 0
&
25\alpha \beta & 6\alpha+5 \beta
&\ldots\\ \vdots & \vdots & \vdots & \vdots & \vdots & \vdots
&
\ddots\end{array}\right).\end{displaymath}
\noindent
For completeness, we note that while in the special case $\alpha=\beta$ the Riordan array is not obviously well-defined, the production matrix is, and it leads in  this special case to the exponential Riordan array
$$\left[\frac{1}{1-\alpha t},\frac{t}{1-\alpha t}\right]$$ which has general element $\binom{n}{k}\frac{n!}{k!}\alpha^{n-k}$. In the case $\alpha=\beta=1$, we get the exponential Riordan array
 $$\left[\frac{1}{1-t}, \frac{t}{1-t}\right]$$ whose inverse is the coefficient array of the Laguerre polynomials \cite{Lah}.

Returning now to the Eulerian polynomials, we set $\alpha=1$ and $\beta=x$, to get
\begin{theorem}
The Eulerian polynomials $P_n(x)$ are the moments of the family of orthogonal polynomials $Q_n(t)$ defined by
$Q_0(t)=1$, $Q_1(t)=t-1$, and
$$Q_n(t)=(t-((n-1)x+n))Q_{n-1}(t)-(n-1)^2 x Q_{n-2}(t).$$
\end{theorem}
\begin{proof}
The initial polynomial terms of the sequence $Q_n(t)$ can be read from the elements of
$$\left[\frac{(1-x)e^{(1-x)t}}{1-xe^{(1-x)t}},\frac{e^{(1-x)t}-1}{1-xe^{(1-x)t}}\right]^{-1}=\left[\frac{1}{1+t},\frac{1}{1-x}\ln\left(\frac{1+t}{1+xt}\right)\right],$$ which begins
\begin{displaymath}\left(\begin{array}{ccccc} 1 & 0 &
0
& 0  & \ldots \\-1 & 1 & 0 & 0  & \ldots \\ 2 & -x-3
& 1 & 0 & \ldots \\ -6 & 2x^2+5x+11 & -3(x+2) & 1  & \ldots \\
\vdots &
\vdots & \vdots & \vdots  &
\ddots\end{array}\right).\end{displaymath}
Hence in particular $Q_0(t)=1$ and $Q_1(t)=t-1$. The three-term recurrence is derived from the production matrix, which in this case is \begin{displaymath}\left(\begin{array}{ccccccc}
1
& 1 & 0 & 0 & 0 & 0 & \ldots \\ x  & 2+x & 1 & 0 & 0 & 0 &
\ldots \\ 0 & 4 x & 3+2 x & 1 & 0 & 0 & \ldots \\ 0 & 0 & 9  x & 4+3x  & 1
&
0 & \ldots \\ 0 & 0 & 0 & 16 x & 5+4 x & 1 & \ldots \\0 & 0  & 0 & 0
&
25 x & 6+5 x
&\ldots\\ \vdots & \vdots & \vdots & \vdots & \vdots & \vdots
&
\ddots\end{array}\right).\end{displaymath}
\end{proof}
\begin{corollary} The sequence of Eulerian polynomials $P_n(x)$ has ordinary generating function given by the continued fraction
$$\cfrac{1}{1-t-
\cfrac{xt^2}{1-(2+x)t-
\cfrac{4xt^2}{1-(3+2x)t-
\cfrac{9xt^2}{1-\cdots}}}}.$$
\end{corollary}
\begin{corollary}
The Hankel transform of the sequence of Eulerian polynomials $P_n(x)$ is given by
$$h_n = x^{\binom{n+1}{2}}\prod_{k=1}^n k!^2.$$
\end{corollary}
\section{The shifted Eulerian polynomials $P_{n+1}(x)$}
For the shifted Eulerian polynomials $P_{n+1}(x)$, we consider the exponential Riordan array
$$[g'(t), f(t)],$$ where
$$g'(t)=\frac{(\alpha-\beta)^2 e^{(\alpha+\beta)t}}{\beta e^{\alpha t}-\alpha e^{\beta t}},$$ where we retain the use of
$g(t)=\frac{(\alpha-\beta)e^{(\alpha-\beta)t}}{\alpha-\beta e^{(\alpha-\beta)t}}$ from the previous section.

When $\alpha=1$ and $\beta=x$, $g'(t)$ generates the shifted sequence $P_{n+1}(x)$.
We then have
\begin{proposition} The production matrix of the exponential Riordan array
$$\left[\frac{(\alpha-\beta)^2 e^{(\alpha+\beta)t}}{\beta e^{\alpha t}-\alpha e^{\beta t}}, \frac{e^{(\alpha-\beta)t}-1}{\alpha-\beta e^{(\alpha-\beta)t}}\right]$$ is tri-diagonal.
\end{proposition}
\begin{proof} As in the previous proposition, we obtain
$$A(t)=f'(\bar{f}(t))=(\alpha t+1)(\beta t+1)=1+(\alpha+\beta)t+ \alpha \beta t^2,$$
where
$$\bar{f}(t)=\frac{1}{\alpha-\beta}\ln\left(\frac{\alpha t+1}{\beta t+1}\right).$$
Then
$$Z(t)=\frac{g''(\bar{f}(t))}{g'(\bar{f}(t))}=(\alpha+\beta)+2\alpha \beta t.$$ The bivariate generating function of the
production matrix is then
$$e^{ty} ((\alpha+\beta)+2\alpha \beta t+(1+(\alpha+\beta)t+ \alpha \beta t^2)y),$$ and hence the
production matrix is tri-diagonal.
\end{proof}
\noindent The production matrix in this case begins
\begin{displaymath}\left(\begin{array}{ccccc} \alpha+\beta & 1 &
0
& 0  & \ldots \\2 \alpha \beta & 2(\alpha+\beta) & 1 & 0  & \ldots \\ 0 & 6\alpha \beta
& 3(\alpha+\beta) & 1 & \ldots \\ 0 & 0 & 12\alpha \beta  & 4(\alpha +\beta)  & \ldots \\
\vdots &
\vdots & \vdots & \vdots  &
\ddots\end{array}\right).\end{displaymath}
In the case $\alpha=\beta$, we obtain the exponential Riordan array
$$\left[ \frac{1}{(1-\alpha t)^2},\frac{t}{1-\alpha t}\right],$$ with $(n,k)$-th element $\binom{n+1}{k+1}\frac{n!}{k!}\alpha^{n-k}$. For $\alpha=\beta=1$ this gives us
$$\left[\frac{1}{(1-t)^2}, \frac{t}{1-t}\right],$$ which is \seqnum{A105278}.

Specializing to the values $\alpha=1$, $\beta=x$, we get the
\begin{theorem} The shifted Eulerian polynomials $P_{n+1}(x)$ are the moments of the family of
orthogonal polynomials $R_n(t)$ given by $R_0(t)=1$, $R_1(t)=t-x-1$, and for $n >1$,
$$R_n(t)=(t-n(1+x))R_{n-1}(t)-n(n-1)x R_{n-2}(t).$$
\end{theorem}
\begin{proof}
The initial terms of the polynomial sequence $R_n(t)$ can be read from the elements of the inverse matrix
$$\left[\frac{(\alpha-\beta)^2 e^{(\alpha+\beta)t}}{\beta e^{\alpha t}-\alpha e^{\beta t}},\frac{e^{(1-x)t}-1}{1-xe^{(1-x)t}}\right]^{-1}=\left[\frac{1}{(1+t)(1+tx)},\frac{1}{1-x}\ln\left(\frac{1+t}{1+xt}\right)\right],$$
which begins
\begin{displaymath}\left(\begin{array}{ccccc} 1 & 0 &
0
& 0  & \ldots \\-x-1 & 1 & 0 & 0  & \ldots \\ 2x^2+2x+2 & -3(x+1)
& 1 & 0 & \ldots \\ -6(x^3+x^2+x+1) & 11x^2+14x+11 & -6(x+1) & 1  & \ldots \\
\vdots &
\vdots & \vdots & \vdots  &
\ddots\end{array}\right).\end{displaymath}
The three-term recurrence is derived from the production matrix, which in this case is \begin{displaymath}\left(\begin{array}{ccccccc}
1+x
& 1 & 0 & 0 & 0 & 0 & \ldots \\ 2x  & 2(1+x) & 1 & 0 & 0 & 0 &
\ldots \\ 0 & 6 x & 3(1+x) & 1 & 0 & 0 & \ldots \\ 0 & 0 & 12x   & 4(1+x)  & 1
&
0 & \ldots \\ 0 & 0 & 0 & 20 x & 5(1+x) & 1 & \ldots \\0 & 0  & 0 & 0
&
30 x & 6(1+x)
&\ldots\\ \vdots & \vdots & \vdots & \vdots & \vdots & \vdots
&
\ddots\end{array}\right).\end{displaymath}
\end{proof}
\begin{corollary} The sequence of shifted Eulerian polynomials $P_{n+1}(x)$ has ordinary generating function given by the continued fraction
$$\cfrac{1}{1-(1+x)t-
\cfrac{2xt^2}{1-2(1+x)t-
\cfrac{6xt^2}{1-3(1+x)t-
\cfrac{12xt^2}{1-\cdots}}}}.$$
\end{corollary}
\begin{corollary} The Hankel transform of the shifted Eulerian polynomials $P_{n+1}(x)$ is given by
$$h_n= (2x)^{\binom{n+1}{2}}\prod_{k=1}^n \binom{k+2}{2}^{n-k}.$$
\end{corollary}
\section{The Eulerian number triangles}
As with the Narayana numbers and their associated number triangles \cite{Narayana}, we can distinguish between three distinct but related triangles of Eulerian numbers.
Thus we have the triangle \seqnum{A173018} \cite{Concrete, Hirzebruch}
\begin{displaymath}\left(\begin{array}{ccccccc} 1 & 0 & 0
& 0 & 0 & 0 & \ldots \\1 & 0 & 0 & 0 & 0 & 0 & \ldots \\
1 & 1 & 0 & 0 & 0 & 0 & \ldots \\ 1 & 4 & 1 & 0
& 0 & 0 & \ldots \\ 1 & 11 & 11 & 1 & 0 & 0 & \ldots
\\1 & 26 & 66 & 26 & 1 & 0 &\ldots\\ \vdots & \vdots & \vdots & \vdots & \vdots
& \vdots & \ddots\end{array}\right) \end{displaymath} of Eulerian numbers
$W_{n,k}$ that obey the recurrence
$$W_{n,k}=(k+1) W_{n-1,k}+(n-k) W_{n-1, k-1}$$ with appropriate boundary conditions, for which the  closed form expression
$$W_{n,k}=\sum_{i=0}^{n-k} (-1)^i \binom{n+1}{i} (n-k-i)^n$$ holds.
We have the reversal of this triangle, which is the triangle \seqnum{A123125} of the coefficients $A_{n,k}$ \cite{Aigner} where
$$A_{n,k}=\sum_{i=0}^k (-1)^i \binom{n+1}{i} (k-i)^n,$$ which begins
\begin{displaymath}\left(\begin{array}{ccccccc} 1 & 0 & 0
& 0 & 0 & 0 & \ldots \\0 & 1 & 0 & 0 & 0 & 0 & \ldots \\
0 & 1 & 1 & 0 & 0 & 0 & \ldots \\ 0 & 1 & 4 & 1
& 0 & 0 & \ldots \\ 0 & 1 & 11 & 11 & 1 & 0 & \ldots
\\0 & 1 & 26 & 66 & 26 & 1 &\ldots\\ \vdots & \vdots & \vdots & \vdots & \vdots
& \vdots & \ddots\end{array}\right), \end{displaymath} and finally we have the Pascal-like triangle of
coefficients $$\tilde{A}_{n,k}=A_{n+1,k+1}=\sum_{i=0}^{k+1} (-1)^i \binom{n+2}{i} (k-i)^{n+1},$$ which begins
\begin{displaymath}\left(\begin{array}{ccccccc} 1 & 0 & 0
& 0 & 0 & 0 & \ldots \\1 & 1 & 0 & 0 & 0 & 0 & \ldots \\
1 & 4 & 1 & 0 & 0 & 0 & \ldots \\ 1 & 11 & 11 & 1
& 0 & 0 & \ldots \\ 1 & 26 & 66 & 26 & 1 & 0 & \ldots
\\1 & 57 & 302 & 302 & 57 & 1 &\ldots\\ \vdots & \vdots & \vdots & \vdots & \vdots
& \vdots & \ddots\end{array}\right). \end{displaymath} This is \seqnum{A008292}.

\noindent We have
$$\tilde{A}_{n,k}=(n-k+1)\tilde{A}_{n-1}{k-1}+(k+1)\tilde{A}_{n-1}{k},$$ with appropriate boundary conditions.
As with the Narayana numbers, each of these triangles has signficant combinatorial applications and it is often important to distinguish one from the other.
\begin{example} The sequence $a_n=\sum_{k=0}^n W_{n,k} 2^k$ is the sequence \seqnum{A000670} of preferential arrangements, or rankings of competitors in a race, with ties \cite{Mendelson}. The sequence
$$b_n=\sum_{k=0}^n A_{n,k}2^k = \sum_{k=0}^n W_{n,n-k} 2^k$$ or \seqnum{A000629} is the sequence of rankings of competitors in a race, with ties and dropouts \cite{MacHale}. Note that from our results above, the sequence $a_n$ has generating function given by
$$\cfrac{1}{1-x-
\cfrac{2x^2}{1-4x-
\cfrac{8x^2}{1-7x-
\cfrac{18x^2}{1-\cdots}}}}.$$
The g.f. of the sequence $a_{n+1}$ is given by
$$\cfrac{1}{1-3x-
\cfrac{4x^2}{1-6x-
\cfrac{12x^2}{1-9x-
\cfrac{24x^2}{1-\cdots}}}}.$$
In this case it happens that $b_n$ is the binomial transform of $a_n$, and hence \cite{Barry_CF} its g.f. has continued fraction expression
$$\cfrac{1}{1-2x-
\cfrac{2x^2}{1-5x-
\cfrac{8x^2}{1-8x-
\cfrac{18x^2}{1-\cdots}}}}.$$
\end{example}
\section{A related ODE}
The form of $f(t)$ above is related to a simple ODE. This arises as follows. In order to have a tri-diagonal production matrix, we need to have an expression of the form
$$A(z)=f'(\bar{f}(z))=1+\mu z+ \nu z^2.$$
Now substituting $z= f(t)$ we obtain
$$f'(\bar{f}(f(t)))=1 + \mu f(t)+ \nu f(t)^2$$ or
$$f'(t)=1+ \mu f(t)+ \nu f(t)^2$$ or
$$ \frac{dy}{dt} = 1+ \mu y +\nu y^2,$$ where $y=f(t)$.
In the Eulerian case above, we have
$$\frac{dy}{dt}=(1+\alpha y)(1+\beta y),$$ with initial condition $y(0)=0$. The form of $y=f(t)$  follows from this variant of the logistic equation.
\section{Appendix: exponential Riordan array}
 The \emph{exponential Riordan group} \cite
{Barry_Pascal, DeutschShap, ProdMat}, is a set of
infinite lower-triangular integer matrices, where each matrix
is defined by a pair
of generating functions $g(x)=g_0+g_1x+g_2x^2+\cdots$ and
$f(x)=f_1x+f_2x^2+\cdots$ where $g_0 \ne 0$ and $f_1\ne 0$. We usually assume that
$$g_0=f_1=1.$$
The associated
matrix is the matrix
whose $i$-th column has exponential generating function
$g(x)f(x)^i/i!$ (the first column being indexed by 0). The
matrix corresponding to
the pair $f, g$ is denoted by $[g, f]$.  The group law is given by \begin{displaymath}
[g,
f]\cdot [h,
l]=[g(h\circ f), l\circ f].\end{displaymath} The identity for
this law is $I=[1,x]$ and the inverse of $[g, f]$ is $[g,
f]^{-1}=[1/(g\circ
\bar{f}), \bar{f}]$ where $\bar{f}$ is the compositional
inverse of $f$.

If $\mathbf{M}$ is the matrix $[g,f]$, and
$\mathbf{u}=(u_n)_{n \ge 0}$
is an integer sequence with exponential generating function
$\mathcal{U}$
$(x)$, then the sequence $\mathbf{M}\mathbf{u}$ has
exponential
generating function $g(x)\mathcal{U}(f(x))$. Thus the row sums
of the array
$[g,f]$ have exponential generating function given by $g(x)e^{f(x)}$ since the sequence
$1,1,1,\ldots$ has exponential generating function $e^x$.

As an element of the group of exponential Riordan arrays, the binomial matrix $\mathbf{B}$ with $(n,k)$-th element $\binom{n}{k}$ is given by
 $\mathbf{B}=[e^x,x]$. By the above, the exponential
generating function of
its row sums is given by $e^x e^x = e^{2x}$, as expected
($e^{2x}$ is the e.g.f. of $2^n$).

To each exponential Riordan array $L=[g,f]$ is associated \cite{ProdMat_0, ProdMat} a matrix $P$ called its \emph{production} matrix, which has
bivariate g.f. given by
$$e^{xy}(Z(x)+A(x)y)$$ where
$$A(x)=f'(\bar{f}(x)), \quad Z(x)=\frac{g'(\bar{f}(x))}{g(\bar{f}(x))}.$$
We have $$P=L^{-1}\bar{L}$$ where $\bar{L}$ \cite{PPWW, Wall} is the matrix $L$ with its top row removed.

\bigskip
\hrule
\bigskip
\noindent 2010 {\it Mathematics Subject Classification}: Primary
11B83; Secondary 33C45, 42C05, 15B36, 15B05, 11C20.

\noindent \emph{Keywords:} Eulerian number, Eulerian polynomial, Euler's triangle, exponential Riordan array, orthogonal polynomials, moments, Hankel transform.

\bigskip
\hrule
\bigskip
\noindent Concerns sequences
\seqnum{A000629},
\seqnum{A000670},
\seqnum{A007318},
\seqnum{A008292},
\seqnum{A105278},
\seqnum{A123125},
\seqnum{A173018}.

\end{document}